\documentclass[11pt,reqno]{amsart}

\usepackage[T1]{fontenc}
\usepackage{lmodern} % clean vector fonts for pdflatex

\usepackage[a4paper,margin=1in]{geometry}
\usepackage{microtype} % better kerning & protrusion
\allowdisplaybreaks % allow long equations to break across pages

\usepackage{amsmath,amssymb,amsthm,mathtools}
\numberwithin{equation}{section}

\usepackage{graphicx}
\usepackage{subcaption}
\usepackage{tikz-cd} % commutative diagrams

\usepackage{enumitem}
\setlist{nosep}
\usepackage{enumitem}

\usepackage[colorlinks=true,linkcolor=blue,citecolor=blue,urlcolor=blue]{hyperref}
\usepackage[nameinlink,capitalise]{cleveref}
\crefname{theorem}{Theorem}{Theorems}
\crefname{lemma}{Lemma}{Lemmas}
\crefname{proposition}{Proposition}{Propositions}
\usepackage[
  backend=biber,
  style=numeric,
  sorting=nyt,
  maxbibnames=9,
  giveninits=true,
  doi=false,
  isbn=false,
  url=false,
  eprint=true
]{biblatex}
\renewbibmacro{in:}{%
  \ifentrytype{article}
    {}
    {\printtext{\bibstring{in}\intitlepunct}}}

\theoremstyle{plain}
\newtheorem{theorem}{Theorem}[section]
\newtheorem{lemma}[theorem]{Lemma}
\newtheorem{proposition}[theorem]{Proposition}

\newtheorem{conjecture}[theorem]{Conjecture}
\theoremstyle{definition}
\newtheorem{definition}[theorem]{Definition}
\theoremstyle{remark}
\newtheorem{remark}[theorem]{Remark}

\newcommand{\Z}{\mathbb{Z}}

\newcommand{\C}{\mathbb{C}}

\DeclareMathOperator{\SL}{SL}

\title{New Ramanujan-type Congruences for Overpartitions Modulo 11 and 13}

\author{XuanLing Wei}
\date{\today}

\begin{document}

\begin{abstract}
In this paper, we establish two new Ramanujan-type congruences for the overpartition function: $\overline{p}(11\times(8n+5))\equiv 0 \pmod{11}$ and 
$\overline{p}(13\times 2^6(8n+7))\equiv 0 \pmod{13}$. The proofs rely on the theory of modular forms. We conjecture potential Ramanujan-type 
congruences for overpartitions modulo 7, 17, 19 and 23. 
\end{abstract}

\maketitle

% ------------------------------
\section{Introduction}
For a positive integer $n$, a partition of $n$ is a nonincreasing sequence of positive integers that sum to $n$. An overpartition is a generalization of a partition, 
in which the first occurrence of each number may be overlined. For example, the overpartitions of 3 are:
\begin{equation*}
  3, \overline{3}, 2+1, \overline{2}+1, 2+\overline{1}, \overline{2}+\overline{1}, 1+1+1, \overline{1}+1+1.
\end{equation*}
The number of partitions of a positive integer $n$ is denoted by $p(n)$, and the number of overpartitions of $n$ is denoted by $\overline{p}(n)$.\par
Ramanujan \cite{ramanujan1919some} discovered the following famous congruences for the partition function $p(n)$:
\begin{equation*}
  \begin{aligned}
    p(5n+4) &\equiv 0 \pmod{5},\\
    p(7n+5) &\equiv 0 \pmod{7},\\
    p(11n+6) &\equiv 0 \pmod{11}.
  \end{aligned}
\end{equation*}
Since then, congruences of the form $p(an+b)\equiv 0 \pmod{m}$, known as Ramanujan-type congruences, have been extensively studied. 
See, for  example, \cite{atkin1967some,ono2000distribution,weaver2001new} for further developments.\par
As a generalization of the partition function, the overpartition function has also attracted considerable interest, particularly with respect to 
Ramanujan-type congruences. Treneer \cite{treneer2006congruences} proved that for a prime $Q\equiv -1 \pmod{5}$,
\begin{equation*}
  \overline{p}(5Q^3 n) \equiv 0 \pmod{5},
\end{equation*}
for all $n$ coprime to $Q$. 
Lovejoy and Osburn \cite{Lovejoy2010QuadraticFA} proved that for an odd prime $Q\equiv -1 \pmod{3}$,
\begin{equation*}
  \overline{p}(3Q^3 n) \equiv 0 \pmod{3},
\end{equation*}
for all $n$ coprime to $Q$.
Nathan C.~Ryan et al. \cite{ryan2024explicit} systematically studied congruences of this type, gave infinite families of congruences modulo $m$ for $m=3,5,7,11$, and finite families of congruences
for $m=13,17,19$.\par
Other Ramanujan-type congruences for overpartitions take the form
\begin{equation*}
\overline{p}(an+b)\equiv 0 \pmod{m},
\end{equation*}
where $a$ is even.
Hirschhorn and Sellers \cite{hirschhorn2005arithmetic} conjectured that
\begin{equation*}
  \overline{p}(40n+35) \equiv 0 \pmod{40}.
\end{equation*}
The first proof was given by Chen and Xia \cite{ChenXia2014Overpartitions}, who showed that
\begin{equation*}
  \overline{p}(40n+35) \equiv 0 \pmod{5},
\end{equation*}
and, combined with the congruence modulo $8$ established earlier by Hirschhorn and Sellers, this result completes the proof of their conjecture modulo $40$.
Later, Liuquan Wang \cite{wang2014another} and Bernard L.~S.~Lin \cite{lin2015new} independently gave further proofs of this conjecture.\par
By numerical computation, we observed several potential Ramanujan-type congruences for overpartitions modulo $7, 11, 13, 17, 19,$ and $23$ with even $a$.
In this paper, we prove two Ramanujan-type congruences for overpartitions modulo $11$ and $13$. Our proofs rely on the theory of modular forms.
\begin{theorem}[Main Theorem]\label{thm:main}
For all $n \geq 0$, the following congruences hold:
\begin{enumerate}[label=\textnormal{(\arabic*)}]
  \item $\overline{p}(11(8n+5)) \equiv 0 \pmod{11}$,
  \item $\overline{p}(13\cdot 2^6(8n+7)) \equiv 0 \pmod{13}$.
\end{enumerate}
\end{theorem}
We adopt the following conventions throughout this paper. 
\begin{definition}
  Let $f,\,g\in \Z[[q]]$ be formal power series with
  \begin{equation*}
    \begin{aligned}
      f&=\sum_{n=0}^{\infty} a(n) q^n,\\
      g&=\sum_{n=0}^{\infty} b(n) q^n.
    \end{aligned}
  \end{equation*}
Let $p$ be a prime, we write $f\equiv g \pmod{p}$ if $a(n)\equiv b(n) \pmod{p}$ for all $n\geq 0$.
\end{definition}
\begin{definition}[q-Pochhammer symbol]
\begin{equation*}
  \begin{aligned}
    (a;q)_{n} &:= \prod_{k=0}^{n-1} (1 - a q^k),\quad n \geq 1,\\
    (a;q)_{\infty} &:= \prod_{k=0}^{\infty} (1 - a q^k).
  \end{aligned}
\end{equation*}
\end{definition}
\begin{lemma}\cite[Cf.~Lemma~1.2]{radu2011congruence}
  \label[lemma]{lem1}
Let $p$ be a prime and $\alpha$ be a positive integer. Then
    \begin{equation*}
    (q;q)_{\infty}^{p^{\alpha}} \equiv (q^p;q^p)_{\infty}^{p^{\alpha-1}} \pmod{p^{\alpha}}.
    \end{equation*}
\end{lemma}

% ------------------------------
\section{Preliminaries}
In this section, we recall the definitions and some known results on modular forms of integral and half-integral weight.

Let $k$ be a positive integer and let $\Gamma \subseteq \SL_2(\Z)$ be a congruence subgroup.
We denote by $\mathcal{M}_k(\Gamma)$ the space of holomorphic modular forms of weight $k$ on $\Gamma$.
Let $\chi$ be a Dirichlet character modulo $N$, where $N$ is a positive integer.
We denote by $\mathcal{M}_k(\Gamma_0(N),\chi)$ the $\chi$-eigenspace of $\mathcal{M}_k(\Gamma_1(N))$,
\begin{equation*}
  \mathcal{M}_k(\Gamma_0(N),\chi) = \{f \in \mathcal{M}_k(\Gamma_1(N)) : f[\gamma]= \chi(d)f \text{ for all } \gamma = \begin{pmatrix} a & b \\ c & d \end{pmatrix} \in \Gamma_0(N)\}. 
\end{equation*}
The vector space $\mathcal{M}_k(\Gamma_1(N))$ decomposes as a direct sum of $\chi$-eigenspaces:
\begin{equation*}
  \mathcal{M}_k(\Gamma_1(N))=\bigoplus_{\chi} \mathcal{M}_k(\Gamma_0(N),\chi).
\end{equation*}
Now let $k$ be an odd positive integer. We denote by $\mathcal{M}_{k/2}(\Gamma_0(4N))$ the space of holomorphic modular forms of weight $\dfrac{k}{2}$ on $\Gamma_0(4N)$,
and by $\mathcal{M}_{k/2}(\Gamma_0(4N),\chi)$ the corresponding $\chi$-eigenspace.
\begin{proposition}\cite[Proposition 2.3]{ono2004web}\label[proposition]{prop:hecke-int}
Let $N$ and $k$ be positive integers, and let
\begin{equation*}
  f(z)=\sum_{n\geq 0} a(n) q^n \in \mathcal{M}_k(\Gamma_0(N),\chi),
  \quad q=e^{2\pi i z}.
\end{equation*}
For a prime $\ell \nmid N$, the action of the Hecke operator $T_{k,N}(\ell)$ on $f$ is given by
\begin{equation*}
  f \mid T_{k,N}(\ell)
  = \sum_{n=0}^{\infty}
    (a(\ell n)+\chi(\ell)\ell^{k-1}a(n/\ell)) q^n,
\end{equation*}
where $a(n/\ell)=0$ if $\ell\nmid n$.
Moreover, $f\mid T_{k,N}(\ell)\in \mathcal{M}_k(\Gamma_0(N),\chi)$.
\end{proposition}
\begin{remark}\label[remark]{rem:hecke-mod}
It is straightforward to see that if
$f=\sum_{n\geq 0} a(n) q^n \in \mathcal{M}_k(\Gamma_0(N),\chi)$ with $k\geq 2$, then
\begin{equation*}
  f \mid T_{k,N}(\ell)
  \equiv \sum_{n\geq 0} a(\ell n) q^n \pmod{\ell}.
\end{equation*}
\end{remark}
We next recall several operators on modular forms of half-integral weight; for details, see \cite[Section~3.2]{ono2004web}.
\begin{proposition}\label[proposition]{prop:half-int-ops}
Let $k$ be a positive odd integer and $N$ a positive integer, and let
\begin{equation*}
  f(z)=\sum_{n\geq 0} a(n) q^n \in \mathcal{M}_{k/2}(\Gamma_0(4N),\chi),
\end{equation*}
where $\chi$ is a Dirichlet character modulo $4N$.
If $d$ is a positive integer, define the Dirichlet character
$\chi_d(n)=\left(\dfrac{4d}{n}\right)$, where $\left(\dfrac{\cdot}{\cdot}\right)$
denotes the Kronecker symbol. Then
\begin{equation*}
  f \mid V(d)
  := \sum_{n\geq 0} a(n) q^{dn}
  \in \mathcal{M}_{k/2}(\Gamma_0(4Nd),\chi_d\chi).
\end{equation*}
If $d\mid N$, then
\begin{equation*}
  f \mid U(d)
  := \sum_{n\geq 0} a(dn) q^n
  \in \mathcal{M}_{k/2}(\Gamma_0(4N),\chi_d\chi).
\end{equation*}
Let $\psi$ be a Dirichlet character with conductor $m$.
The twist of $f$ by $\psi$ is defined by
\begin{equation*}
  f\otimes \psi
  := \sum_{n\geq 0} \psi(n)a(n) q^n
  \in \mathcal{M}_{k/2}(\Gamma_0(4Nm^2),\chi\psi^2).
\end{equation*}
All of the operators above send cusp forms to cusp forms.
\end{proposition}
\begin{proposition}
  \label[proposition]{prop.1}
Let $k$, $N$ and $f$ be as in the previous proposition. Let $d, A, B$ be positive integers with $A$ coprime to $B$.
Define
\begin{equation*}
  g(z)=\sum_{n\geq 0} a(d(An+B))q^{An+B}.
\end{equation*}
Then  $g(z) \in \mathcal{M}_{k/2}(\Gamma_1(4NdA^2))$.
\end{proposition}
\begin{proof}
    Since we have the natural inclusion $\mathcal{M}_{k/2}(\Gamma_0(4N),\chi)\subseteq\mathcal{M}_{k/2}(\Gamma_0(4Nd),\chi)$, 
    we may view $f$ as an element of $\mathcal{M}_{k/2}(\Gamma_0(4Nd),\chi)$. Applying the operator $U(d)$ to $f$, we obtain
    \begin{equation*}
        f(z)|U(d)=\sum_{n\geq 0} a(dn)q^n\in \mathcal{M}_{k/2}(\Gamma_0(4Nd),\chi_d\chi).
    \end{equation*}
    When we twist $f(z)|U(d)$ by a Dirichlet character $\psi$ modulo $A$, we have
    \begin{equation*}
        (f(z)|U(d))\otimes \psi=\sum_{n\geq 0} \psi(n)a(dn)q^n\in \mathcal{M}_{k/2}(\Gamma_0(4NdA^2), \chi_d\chi\psi^2).
    \end{equation*}
    Denote the group of Dirichlet characters modulo $A$ by $\widehat{(\Z/A\Z)^\times}$, and consider
    \begin{equation*}
        g(z)=\frac{1}{\varphi(A)}\sum_{\psi \in \widehat{(\Z/A\Z)^\times}}\overline{\psi}(B)(f(z)|U(d))\otimes \psi,
    \end{equation*}
    where $\varphi$ is the Euler totient function.
    Since $A$ is coprime to $B$, we have the orthogonal relation of Dirichlet characters: 
    \begin{equation*}
    \frac{1}{\varphi(A)}\sum_{\psi \in \widehat{(\Z/A\Z)^\times}}\overline{\psi}(B)\psi(n)=
        \begin{cases}
        1 & \text{if } n \equiv B \pmod{A},\\
        0 & \text{otherwise}.
        \end{cases}
    \end{equation*}
    Therefore, $g(z)=\sum_{n\geq 0} a(d(An+B))q^{An+B}$, which is a linear combination of modular forms in $\mathcal{M}_{k/2}(\Gamma_1(4NdA^2))$. This completes the proof.
\end{proof}
\begin{remark}
  \label[remark]{rem2}
  The construction above is also valid for integral weight modular forms.
  Especially, suppose that every $\psi \in \widehat{(\Z/A\Z)^\times}$ is a real character, i.e. $\psi^2=\chi_0$ is the trivial character.
  Then $(f(z)\mid U(d)) \otimes \psi \in \mathcal{M}_{k/2}(\Gamma_0(4NdA^2),\chi_d\chi)$ for all $\psi$. Consequently, $g(z)\in \mathcal{M}_{k/2}(\Gamma_0(4NdA^2),\chi_d\chi)$.
\end{remark}
\begin{theorem}[Sturm's theorem\cite{sturm2006congruence}]
  \label[theorem]{sturm}
  \leavevmode
  Let $\Gamma$ be a congruence subgroup, and $p$ be a prime. 
  Assume that $f$ and $g$ are two holomorphic modular forms in $\mathcal{M}_{k}(\Gamma)$ with Fourier expansions:
  \begin{equation*}
    \begin{aligned}
    f=\sum_{n=0}^{\infty} a(n)q^n \in \Z[[q]],\\
    g=\sum_{n=0}^{\infty} b(n)q^n \in \Z[[q]].
  \end{aligned}
  \end{equation*}
  If $a(n) \equiv b(n) \pmod{p}$ for all $n\leq \dfrac{k}{12}[\mathrm{SL}_2(\Z):\Gamma]$, then $f\equiv g \pmod{p}$.
\end{theorem}
\begin{remark}
  For any half-integral weight modular form $f$, the square $f^2$ has integral weight. Moreover, $f\equiv 0 \pmod{p}$ if and only if $f^i \equiv 0 \pmod{p}$ for some positive integer $i$,
  and hence Sturm's theorem applies to half-integral weight modular forms.
\end{remark}
Next, we recall two important functions in the theory of modular forms.
\begin{definition}
  The Dedekind eta function, denoted by $\eta(z)$, is defined by
\begin{equation*}
\eta(z):=q^{1/24}(q;q)_{\infty}, \quad q=e^{2\pi i z}.
\end{equation*}
An eta-quotient $f$ is a function of the form
\begin{equation*}
  f(z)=\prod_{\delta \mid N} \eta^{r_{\delta}}(\delta z),
\end{equation*}
where $\delta$ and $N$ are positive integers, and each $r_{\delta}$ is an integer.
\end{definition}
\begin{definition}
    \label[definition]{de1}
  The Ramanujan theta function $\phi(q)$ is defined by
\begin{equation*}
  \phi(q) = \sum_{n=-\infty}^{\infty} q^{n^2}.
\end{equation*}
We write the expansion of $\phi(q)^m$ as
\begin{equation*}
  \phi(q)^{m}=\sum_{n\geq 0} r_{m}(n) q^n.
\end{equation*}
Ramanujan theta function can be written in terms of eta-quotients as
\begin{equation*}
  \phi(q) = \frac{\eta(2z)^5}{\eta(z)^2 \eta(4z)^2} = \frac{(q^2;q^2)_{\infty}^5}{(q;q)_{\infty}^2 (q^4;q^4)_{\infty}^2}.
\end{equation*}
Moreover, $\phi(q)\in \mathcal{M}_{1/2}(\Gamma_0(4))$.
\end{definition}
It is also well known that
\begin{equation*}
  \phi(-q) = \frac{\eta(2z)}{\eta(z)^2}.
\end{equation*}
Thus, the generating function for $\overline{p}(n)$ can be written as \cite{corteel2004overpartitions}
\begin{equation*}
  \sum_{n\geq 0} \overline{p}(n) q^n = \frac{(q^2;q^2)_{\infty}}{(q;q)_{\infty}^2} = \frac{1}{\phi(-q)}.
\end{equation*}
Define
  \begin{equation*}
    F(q):=\frac{\eta(4z)^8}{\eta(2z)^4}\in \mathcal{M}_{2}(\Gamma_0(4)).
  \end{equation*}
We have the following structure theorem.
\begin{theorem}\cite[Theorem 1.49(2)]{ono2004web}
  \label[theorem]{thm1}
  Let 
  \begin{equation*}
    \psi_k := 
    \begin{cases}
      \chi_0 & \text{ if } k\in 2\Z \text{ or } k\in 1/2+\Z,\\
      \chi_{-4} & \text{ if } k\in 1+2\Z,
    \end{cases}
  \end{equation*}
  where $\chi_0$ is the trivial character modulo 4, and $\chi_{-4}$ is the only non-trivial character modulo 4.
  As a graded algebra over $\C$, we have:
  \begin{equation*}
    \bigoplus_{k\in \frac{1}{2}\Z} \mathcal{M}_{k}(\Gamma_0(4),\psi_k) \cong \C[\phi(q),F(q)].
  \end{equation*}
\end{theorem}
% ------------------------------
\section{Main results}
\begin{proof}[Proof of Theorem~\ref{thm:main}(1)]
\begin{equation*}
    \sum_{n\geq0}\overline{p}(n)q^n=\frac{(q^2;q^2)_{\infty}}{(q;q)^{2}_{\infty}}=\frac{1}{\phi(-q)}
\end{equation*}
Replacing $q$ by $-q$, we have
\begin{equation*}
    \sum_{n\geq0}\overline{p}(n)(-q)^n=\frac{1}{\phi(q)}.
\end{equation*}
Multiplying both sides by $\phi(q)^{11}$, 
\begin{equation*}
    \phi(q)^{11}\sum_{n\geq0}\overline{p}(n)(-q)^n=\phi(q)^{10}.
\end{equation*}
By \cref{lem1,de1}, we have $\phi(q)^{11}\equiv \phi(q^{11}) \pmod{11}$. Thus,
\begin{equation*}
    \phi(q^{11})\sum_{n\geq0}\overline{p}(n)(-q)^n \equiv \phi(q)^{10}=\sum_{n\geq0}r_{10}(n)q^{n} \pmod{11}.
\end{equation*}
Collecting the terms of the form $q^{11n}$ on both sides, we have
\begin{equation*}
    \phi(q^{11})\sum_{n\geq0}\overline{p}(11n)(-q)^{11n} \equiv \sum_{n\geq0}r_{10}(11n)q^{11n} \pmod{11}.
\end{equation*}
Replacing $q^{11}$ by $q$, we have
\begin{equation*}
    \phi(q)\sum_{n\geq0}\overline{p}(11n)(-q)^{n} \equiv \sum_{n\geq0}r_{10}(11n)q^{n} \pmod{11}.
\end{equation*}
Up to this point, the above construction follows \cite{wang2014another}. \par
By \cref{rem:hecke-mod}, we have $\sum_{n\geq0}r_{10}(11n)q^{n}\equiv \phi(q)^{10}\mid T_{5,4}(11) \pmod{11}$.
Since $\phi(q)^{10}\in \mathcal{M}_{5}(\Gamma_0(4),\chi_{-4})$, it follows from \cref{prop:hecke-int} that $\phi(q)^{10}\mid T_{5,4}(11)\in \mathcal{M}_{5}(\Gamma_0(4),\chi_{-4})$.
By \cref{thm1}, we may choose a basis of $\mathcal{M}_{5}(\Gamma_0(4),\chi_{-4})$ as follows:
\begin{equation*}
  \begin{aligned}
    &f_0 = F(q)^2\phi(q)^{2} = q^2 + \cdots,\\
    &f_1 = F(q)\phi(q)^{6} = q + 12 q^2 + \cdots,\\
    &f_2 = \phi(q)^{10} = 1 + 20 q + 180 q^2 + \cdots.
  \end{aligned}
\end{equation*}
Expressing $\phi(q)^{10}\mid T_{5,4}(11)$ in terms of the basis above, and reducing modulo 11, we obtain
\begin{equation*}
  \phi(q)^{10}\mid T_{5,4}(11) \equiv \phi(q)\sum_{n\geq0}\overline{p}(11n)(-q)^{n} \equiv F(q)\phi(q)^{6} + \phi(q)^{10} \pmod{11}.
\end{equation*}
Since $\Z/11\Z[[q]]$ is an integral domain, canceling $\phi(q)$ on both sides, we have
\begin{equation*}
  \sum_{n\geq0}\overline{p}(11n)(-q)^{n} \equiv F(q)\phi(q)^{5} + \phi(q)^{9} \pmod{11}.
\end{equation*}
Here $F(q)\phi(q)^{5}+ \phi(q)^{9} \in \mathcal{M}_{9/2}(\Gamma_0(4))$. \par
Let $F(q)\phi(q)^{5} + \phi(q)^{9}=\sum_{n=0}^{\infty} a(n) q^n$. To prove $\overline{p}(11(8n+5))\equiv 0 \pmod{11}$, 
it suffices to prove that $a(8n+5)\equiv 0 \pmod{11}$. Since every $\psi \in \widehat{(\Z/8\Z)^\times}$ 
is a real character, by \cref{prop.1} and \cref{rem2}, we have
\begin{equation*}
  \sum_{n=0}^{\infty} a(8n+5) q^{8n+5} \in \mathcal{M}_{9/2}(\Gamma_0(256)).
\end{equation*}
By \cref{sturm}, it suffices to check that $a(8n+5)\equiv 0 \pmod{11}$ for $8n+5 \leq 289$. By direct computation, we complete the proof. 
\end{proof}
\begin{proof}[Proof of Theorem~\ref{thm:main}(2)]
Consider $\phi(q)^{13}\sum_{n\geq 0}\overline{p}(n)(-q)^n=\phi(q)^{12}$. Applying the same argument as in the proof above, we have
\begin{equation*}
  \phi(q)\sum_{n\geq 0}\overline{p}(13n)(-q)^{n} \equiv \sum_{n\geq 0}r_{12}(13n)q^{n}\equiv \phi(q)^{12}\mid T_{6,4}(13) \pmod{13}.
\end{equation*}
Since $\phi(q)^{12}\in \mathcal{M}_{6}(\Gamma_0(4))$, it follows from \cref{prop:hecke-int} that $\phi(q)^{12}\mid T_{6,4}(13)\in \mathcal{M}_{6}(\Gamma_0(4))$. We may choose a basis of $\mathcal{M}_{6}(\Gamma_0(4))$ as follows:
\begin{equation*}
  \begin{aligned}
    &f_0 = F(q)^3 = q^3 + \cdots,\\
    &f_1 = F(q)^2\phi(q)^4 = q^2 + 8 q^3 + \cdots,\\
    &f_2 = F(q)\phi(q)^8 = q + 16 q^2 + 116 q^3 + \cdots,\\
    &f_3= \phi(q)^{12} = 1 + 24 q + 264 q^2 + 1760 q^3 + \cdots.
  \end{aligned}
\end{equation*}
Expressing $\phi(q)^{12}\mid T_{6,4}(13)$ in terms of the basis above, and reducing modulo 13, we obtain
\begin{equation*}
  \phi(q)\sum_{n\geq 0}\overline{p}(13n)(-q)^{n} \equiv \phi(q)^{12} \mid T_{6,4}(13) \equiv F(q)^2\phi(q)^4 + 4 F(q)\phi(q)^8 + \phi(q)^{12} \pmod{13}.
\end{equation*}
Since $\Z/13\Z[[q]]$ is an integral domain, canceling $\phi(q)$ on both sides, we have
\begin{equation*}
  \sum_{n\geq 0}\overline{p}(13n)(-q)^{n} \equiv F(q)^2\phi(q)^3 + 4 F(q)\phi(q)^7 + \phi(q)^{11} \pmod{13}.
\end{equation*}
Here $F(q)^2\phi(q)^3 + 4 F(q)\phi(q)^7 + \phi(q)^{11} \in \mathcal{M}_{11/2}(\Gamma_0(4))$.\par
Let $F(q)^2\phi(q)^3 + 4 F(q)\phi(q)^7 + \phi(q)^{11}=\sum_{n\geq 0} b(n)q^n$. To prove $\overline{p}(13\times 2^6(8n+7))\equiv 0 \pmod{13}$,
it suffices to prove that $b(2^6(8n+7))\equiv 0 \pmod{13}$. When using \cref{prop.1}, we apply the operator $U(2)$ six times instead of applying $U(2^6)$ directly. 
Since every $\psi \in \widehat{(\Z/8\Z)^\times}$ is a real character, by \cref{prop.1} and \cref{rem2}, we have 
\begin{equation*}
\sum_{n\geq 0} b(2^6(8n+7))q^{8n+7} \in \mathcal{M}_{11/2}(\Gamma_0(512),\chi_2),
\end{equation*}
where $\chi_2(x)=\left(\dfrac{2}{x}\right)$. By \cref{sturm}, it suffices to check that $b(2^6(8n+7))\equiv 0 \pmod{13}$ for $8n+7 \leq 705$. 
By direct computation, we complete the proof. 
\end{proof}
% ------------------------------
\section{Closing remarks}
Numerical computations suggest that the overpartition function $\overline{p}(n)$ satisfies the following conjectural Ramanujan-type congruences:
\begin{equation*}
  \overline{p}(17\times 11^2 n) \equiv 0 \pmod{17},
  \quad \text{where } n \equiv 3 \pmod{8}
  \text{ and }
  \left(\frac{n}{11}\right) = -1.
\end{equation*}
Equivalently, $\overline{p}(17\times 11^2(88n+m)) \equiv 0 \pmod{17}$, for all $n\geq 0$ and $m\in\{19,35,43,51,83\}$.
\begin{equation*}
  \overline{p}(23\times 13^2 n) \equiv 0 \pmod{23},
  \quad \text{where } n \equiv 5 \pmod{8}
  \text{ and }
  \left(\frac{n}{13}\right) = 1.
\end{equation*}
Equivalently, $\overline{p}(23\times 13^2(104n+l)) \equiv 0 \pmod{23}$, for all $n\geq 0$ and $l\in\{29,53,61,69,77,101\}$.\par
These congruences are expected to be provable by the methods developed in this paper. Indeed, we have
\begin{equation*}
  \sum_{n\geq 0}\overline{p}(17n)(-q)^{n}\equiv \sum_{n\geq 0}c(n)q^{n} = 13 F(q)^2\phi(q)^7 + 13 F(q)\phi(q)^{11} + \phi(q)^{15} \pmod{17},
\end{equation*}
\begin{equation*}
  \begin{aligned}
  \sum_{n\geq 0}\overline{p}(23n)(-q)^{n}\equiv \sum_{n\geq 0}d(n)q^{n} = &20 F(q)^5\phi(q) + 17 F(q)^4\phi(q)^5 + 14 F(q)^3\phi(q)^9 +\\
                                                 &5 F(q)^2\phi(q)^{13} + 9 F(q)\phi(q)^{17} + \phi(q)^{21} \pmod{23}. 
  \end{aligned}
\end{equation*}
We have $\sum_{n\geq 0}c(n)q^{n}\in \mathcal{M}_{15/2}(\Gamma_0(4))$ and $\sum_{n\geq 0}d(n)q^{n}\in \mathcal{M}_{21/2}(\Gamma_0(4))$. By \cref{prop.1}, it follows that
\begin{equation*}
 \sum_{n\geq 0}c(11^2(88n+m))q^{88n+m}\in \mathcal{M}_{15/2}(\Gamma_1(340736)),
\end{equation*}
and 
\begin{equation*}
  \sum_{n\geq 0}d(13^2(104n+l))q^{104n+l}\in \mathcal{M}_{21/2}(\Gamma_1(562432)).
\end{equation*}
By \cref{sturm}, it suffices to verify that
\begin{equation*}
    c(11^2(88n+m))\equiv 0 \pmod{17} \text{ for } n\leq 1{,}226{,}649{,}601,
\end{equation*}
and 
\begin{equation*}
d(13^2(104n+l))\equiv 0 \pmod{23} \text{ for } n\leq 3{,}968{,}520{,}193.
\end{equation*}
This verification requires a substantial amount of computation and has not been completed.\par
We also found some other conjectural Ramanujan-type congruences for overpartitions which cannot be proved by the methods in this paper.
\begin{conjecture}
  \begin{equation*}
  \overline{p}(2^4 n) \equiv 0 \pmod{7},
  \quad \text{where } n \equiv 3 \pmod{8}
  \text{ and }
  \left(\frac{n}{7}\right)=1.
\end{equation*}
Equivalently, $\overline{p}(2^4(56n+k)) \equiv 0 \pmod{7}$ for all $n\geq 0$ and $k\in\{11,43,51\}$.
\end{conjecture}
\begin{conjecture}
  \begin{equation*}
  \overline{p}(17^2\cdot 4n) \equiv 0 \pmod{19},
  \quad \text{where } n \equiv 3 \pmod{8},\ 
  \left(\frac{n}{17}\right)=-1,\ 
  \text{and }\left(\frac{n}{19}\right)=1.
\end{equation*}
It involves $72$ residue classes modulo $2584$; we do not list all of them here.
\end{conjecture}
We hope to address these conjectures in future work.

% ------------------------------
\printbibliography

@article{corteel2004overpartitions,
  title={Overpartitions},
  author={Corteel, Sylvie and Lovejoy, Jeremy},
  journal={Transactions of the American Mathematical Society},
  volume={356},
  number={4},
  pages={1623--1635},
  year={2004}
}

@book{ono2004web,
  title={The Web of Modularity: Arithmetic of the Coefficients of Modular Forms and $ q $-series: Arithmetic of the Coefficients of Modular Forms and Q-series},
  author={Ono, Ken},
  number={102},
  year={2004},
  publisher={American Mathematical Soc.}
}

@article{wang2014another,
  title={Another Proof of a Conjecture by Hirschhorn and Sellers on Overpartitions.},
  author={Wang, Liuquan},
  journal={J. Integer Seq.},
  volume={17},
  number={9},
  pages={14--9},
  year={2014}
}

@article{ryan2024explicit,
  title={Explicit families of congruences for the overpartition function},
  author={Ryan, Nathan C and Sirolli, Nicol{\'a}s and Villegas-Morales, Jean Carlos and Zheng, Qi-Yang},
  journal={The Ramanujan Journal},
  volume={65},
  number={4},
  pages={1631--1649},
  year={2024},
  publisher={Springer}
}

@inproceedings{ramanujan1919some,
  title={Some properties of p (n), the number of partitions of n},
  author={Ramanujan, Srinivasa},
  booktitle={Proc. Cambridge Philos. Soc},
  volume={19},
  pages={207--210},
  year={1919}
}

@article{lin2015new,
  title={A new proof of a conjecture of Hirschhorn and Sellers on overpartitions},
  author={Lin, Bernard LS},
  journal={The Ramanujan Journal},
  volume={38},
  number={1},
  pages={199--209},
  year={2015},
  publisher={Springer}
}

@article{hirschhorn2005arithmetic,
  title={Arithmetic relations for overpartitions},
  author={Hirschhorn, Michael D and Sellers, James A},
  journal={J. Combin. Math. Combin. Comput},
  volume={53},
  number={65-73},
  pages={1},
  year={2005}
}

@article{ChenXia2014Overpartitions,
  author    = {Chen, William Y. C. and Xia, Ernest X. W.},
  title     = {Proof of a conjecture of Hirschhorn and Sellers on overpartitions},
  journal   = {Acta Arithmetica},
  volume    = {163},
  year      = {2014},
  pages     = {59--69},
}

@article{ono2000distribution,
  title={Distribution of the partition function modulo m},
  author={Ono, Ken},
  journal={Annals of Mathematics},
  volume={151},
  number={1},
  pages={293--307},
  year={2000},
  publisher={JSTOR}
}

@article{treneer2006congruences,
  title={Congruences for the coefficients of weakly holomorphic modular forms},
  author={Treneer, Stephanie},
  journal={Proceedings of the London Mathematical Society},
  volume={93},
  number={2},
  pages={304--324},
  year={2006},
  publisher={Cambridge University Press}
}

@inproceedings{Lovejoy2010QuadraticFA,
  title={Quadratic Forms and Four Partition Functions Modulo 3},
  author={Jeremy Lovejoy and Robert Osburn},
  booktitle={Integers},
  year={2010}
}

@inproceedings{sturm2006congruence,
  title={On the congruence of modular forms},
  author={Sturm, Jacob},
  booktitle={Number Theory: A Seminar held at the Graduate School and University Center of the City University of New York 1984--85},
  pages={275--280},
  year={2006},
  organization={Springer}
}

@article{radu2011congruence,
  title={Congruence properties modulo 5 and 7 for the pod function},
  author={Radu, Silviu and Sellers, James A},
  journal={International Journal of Number Theory},
  volume={7},
  number={08},
  pages={2249--2259},
  year={2011},
  publisher={World Scientific}
}

@article{atkin1967some,
  title={Some properties of p (n) and c (n) modulo powers of 13},
  author={Atkin, AOL and O'brien, JN},
  journal={Transactions of the American Mathematical Society},
  volume={126},
  number={3},
  pages={442--459},
  year={1967},
  publisher={JSTOR}
}

@article{weaver2001new,
  title={New congruences for the partition function},
  author={Weaver, Rhiannon L},
  journal={The Ramanujan Journal},
  volume={5},
  number={1},
  pages={53--63},
  year={2001},
  publisher={Springer}
}

% =============================================================
%               BibTeX SWITCH (if required by a journal)
% -------------------------------------------------------------
% 1) Comment out the biblatex lines above (\usepackage[...]{biblatex} and \addbibresource{...}).
% 2) Add the following two lines BEFORE \begin{document}:
%    \usepackage[numbers]{natbib}
%    \bibliographystyle{abbrvnat} % or plain, alpha, etc.
% 3) Replace \printbibliography with:
%    \bibliography{references}
% 4) Run: pdflatex → bibtex → pdflatex → pdflatex
% =============================================================

% =============================================================
%          Sample references.bib (copy into a .bib file)
% -------------------------------------------------------------
% 
% =============================================================
\end{document}